\documentclass{amsart}

\usepackage{amsmath, amsthm, amssymb}
\usepackage{epsfig}
\usepackage{graphicx}
\usepackage{array}
\usepackage{amsmath}
\usepackage{amsfonts}
\usepackage{amsfonts}
\usepackage{amsmath}%
\usepackage{amsfonts}%
\usepackage{amssymb}%
\usepackage{graphicx}

\usepackage{amsfonts}
\usepackage{amssymb}
\usepackage{amsmath,amsxtra,amssymb}
\numberwithin{equation}{section}

\numberwithin{equation}{section}

\newtheorem{theorem}{Theorem}
\newtheorem{lemma}[theorem]{Lemma}

\newtheorem{corollary}[theorem]{Corollary}

\newcommand{\id}{{\iota}}
\newcommand{\ot}{{\otimes}}
\newcommand{\om}{{\omega}}
\newcommand{\tp}{{\widehat{\otimes}}}
\newcommand{\vtp}{{\overline{\otimes}}}

\newcommand{\B}{{\mathcal{B}}}

\newcommand{\fee}{{\varphi}}

\newcommand{\G}{\mathbb G}
\newcommand{\LL}{{\mathcal{L}^{\infty}(\G)}}
\newcommand{\LO}{{\mathcal{L}^{1}(\G)}}
\newcommand{\LT}{{\mathcal{L}^{2}(\G)}}
\newcommand{\MG}{{\mathcal{M}(\G)}}

\newcommand{\CZ}{{\mathcal{C}_0(\G)}}
\newcommand{\BZ}{{\mathcal{B}_0(\LT)}}

\title
{Compact Operators in Regular LCQ Groups}
\author{Mehrdad Kalantar}
\address{ School of Mathematics and Statistics,
Carleton University, Ottawa, ON, Canada}
\email{mkalanta@math.carleton.ca}


\begin{document}

\begin{abstract}
We show that a regular locally compact quantum group $\G$ is discrete
if and only if $\LL$ contains non-zero compact operators on $\LT$.
As a corollary we classify all discrete quantum groups among
regular locally compact quantum groups $\G$ where $\LO$ has the Radon--Nikodym property.
\end{abstract}
\maketitle

It is known that for a locally compact group $G$,
the following are equivalent: $(i)$ $G$ is discrete, $(ii)$ $L^\infty(G)$ contains a compact operator on $L^2(G)$,
$(iii)$ $L^1(G)$ has the Radon--Nikodym property, $(iv)$ the von Neumann algebra $L^\infty(G)$ is purely atomic
(cf. \cite{Diest2} and \cite{Taylor}).

In the general setting of locally compact quantum groups $\G$, it is known that $(i)$ implies other
properties; $(iii)$ and $(iv)$ are equivalent; and there are examples of $\G$ that satisfy $(iii)$, but not $(i)$
(cf.  \cite{Taylor} and \cite{Wor}).

In this paper we investigate the relations between $(ii)$ and other above properties.
We prove that in the case of regular locally compact quantum groups, $(ii)$ implies
$(i)$, whence $(iii)$ and $(iv)$. Moreover we classify regular locally compact quantum groups
which satisfy $(iv)$, and (or but not) $(ii)$.

First, let us recall some definitions and preliminary results that we will be using in this paper.
For more details on locally compact quantum groups we refer the reader to \cite{KV}.

A {\it locally compact quantum group} $\G$ is a
quadruple $(\LL, \Gamma, \varphi, \psi)$, where $\LL$ is a
von Neumann algebra,
$\Gamma: \LL\to \LL \bar\otimes \LL$
is a co-associative co-multiplication, i.e. a unital injective *-homomorphism, satisfying
\[
(\Gamma\otimes\id)\,\Gamma\,=\,(\id\otimes\Gamma)\,\Gamma\,,
\]
and $\varphi$ and $\psi$ are (normal faithful semi-finite) left and right
invariant weights on $\LL$, that is
\[
\fee\big((\om\otimes\id)\,\Gamma(x)\big)\,=\,\fee(x)\om(1)
\]
for all $\om\in\LO^+$ and $x\in\LL^+$ where $\fee(x)<\infty$, and
\[
\psi\big((\id\otimes\om)\,\Gamma(x)\big)\,=\,\psi(x)\om(1)
\]
for all $\om\in\LO^+$ and $x\in\LL^+$ where $\psi(x)<\infty$.
We denote by $\LT$ the GNS Hilbert space of $\fee$. Then one obtains
two distinguished unitary operators
$W\in \LL\vtp\B(\LT)$ and $V\in\B(\LT)\vtp\LL$,
called the left and right \emph{fundamental unitaries}, 
which satisfy the \emph{pentagonal relation},
and such that the co-multiplication $\Gamma$ on $\LL$ can be expressed as
\begin{equation*} 
\Gamma(x) \,=\, W^{*}(1\otimes x)W\,=\,V(x\otimes 1)V^*
~~\quad(x \in \LL)\,.
\end{equation*}
The \emph{reduced quantum group $C^*$-algebra}
\[
\overline{\big\{\,(\id\otimes\om)\,W\,:\,\om\in\B(\LT)_*\,\big\}}^{\|\cdot\|}\,=\,
\overline{\big\{\,(\om\otimes\id)\,V\,:\,\om\in\B(\LT)_*\,\big\}}^{\|\cdot\|}
\]
is denoted by $\CZ$, which
is a weak$^*$ dense $C^*$-subalgebra of $\LL$.
Let $\MG$ denote the dual space $\CZ^{*}$.
There exists a completely contractive multiplication on $\MG$ given by
the convolution
\[
\star : \MG\tp \MG\ni \mu\ot \nu \,\longmapsto\, \mu \star \nu 
= \mu (\id\otimes \nu)\Gamma = \nu (\mu \otimes \id)\Gamma
\in \MG
\]
such that  $\MG$ contains $\LO:= \LL_*$ as a norm closed two-sided ideal.
Therefore, for each $\mu\in\MG$, we obtain a pair of completely bounded maps
\begin{equation*}
f\, \longmapsto \, \mu \star f\ \  ~ \mbox{and } ~ \ \ f \,\longmapsto\, f \star \mu
\end{equation*}
on $\LO$ through the left and right convolution products of $\MG$.
The adjoint maps give the convolution actions $x\mapsto \mu\star x$ and $x\mapsto x\star\mu$
that are normal completely bounded maps on $\LL$
(note that our notation for the convolution actions is opposite to
the more commonly used (e.g. \cite{HNR}), where $\mu\star x$ is denoted by $x\star \mu$).

For a Hilbert space $H$, we denote by $\B(H)$,
and $\B_0(H)$ the spaces of all bounded operators, and compact operators on $H$, respectively.

A locally compact quantum group $\G$ is said to be {\it regular} if
the norm-closed linear span of $\{\,(\id\otimes\om)\,(\Sigma V)\,:\,\om\in\B(\LT)_*\,\}$ equals $\B_0(\LT)$,
where $\Sigma$ denotes the flip operator on $\LT\otimes\LT$.
All Kac algebras, as well as discrete and compact quantum groups are regular \cite{BS}.

It follows from \cite[Theorem 3.1]{HNR} that for $a\in\B_0(\LT)$ and $\om\in\B(\LT)_*$, we have
\[
(\id\otimes\om)\,\left(W^*\,(1\otimes a)W\right)\,\in\,\CZ~~~~~~~~~~~~\ \ \ \ \text{and} \ \ \ \ ~~~~~~~~~~~~
(\om\otimes\id)\,\left(V\,(a\otimes 1)V^*\right)\,\in\,\CZ\,,
\]
and also it is proved in \cite[Corollary 3.6]{HNR} that if $\G$ is regular, then
\[
(\om\otimes\id)\,\left(W^*\,(1\otimes a)W\right)\,\in\,\B_0(\LT)~~~~~~~~~~~~\ \ \ \ \text{and} \ \ \ \ ~~~~~~~~~~~~
(\id\otimes\om)\,\left(V\,(a\otimes 1)V^*\right)\,\in\,\B_0(\LT)\,.
\]
Therefore, if $\G$ is regular and $a\in\LL\cap\B_0(\LT)$ then
\begin{equation}\label{1313}
f\star a\,,\ a\star f\,\in\,\CZ\cap\B_0(\LT)
\end{equation}
for all $f\in\LO$.

\par
The following is the main result of the paper.
\begin{theorem}\label{main}
Let $\G$ be a regular locally compact quantum group. If
$$\LL\,\cap\,\BZ\, \neq\, \{0\}\,,$$
then $\G$ is discrete.
\end{theorem}
Note that the converse of this theorem is also true; it in fact follows from
the structure theory of discrete quantum groups (cf. \cite{Wor}).

\par
We break down the proof of the theorem into few lemmas that follow.

\begin{lemma}\label{P.faithful}
Let $0\leq\mu\in \MG$ be non-zero.
Then the convolution map $x \mapsto x\star \mu$ is faithful on $\LL^+$.
\end{lemma}
\begin{proof}
Let $\psi$ be the right Haar weight of $\G$, then we have
\begin{eqnarray*}
\psi (x\star\mu) &=& \psi (x\star \mu)\,\om(1) \,=\,  \psi \big((\id\otimes \om)\,\Gamma (x\star\mu)\big)\\
&=& \psi\big((\id\otimes (\om\star\mu))\, \Gamma(x)\big)
\,=\, \psi(x)\,\om(1)\,\|\mu\|
\end{eqnarray*}
for all $x \in \LL^{+}$ and $\om\in\LO^+$, and since $\psi$ is faithful, the lemma follows.
\end{proof}

\begin{lemma}\label{07}
The von Neumann algebra $\LL$ is purely atomic.
\end{lemma}
\begin{proof}
Since $\LL \,\cap\,\B_0(\LT)\,\neq\,\{0\}$, it follows that $\LL$ contains a non-zero minimal projection.
Suppose that $\{e_\alpha\}$ is a maximal
family of mutually orthogonal minimal projections in $\LL$.
Set $e_0 = 1 - \sum_{\alpha} e_\alpha$, and let $a\in\LL \,\cap\,\B_0(\LT)$
be non-zero and positive.
Moreover, suppose that $\om_0\in\LO^+$
is such that supp$(\om_0)\leq e_0$. Then, since $e_0$ does not dominate
any non-zero minimal projection, it follows that $e_0 \left(\LL \,\cap\,\B_0(\LT)\right) e_0 = 0$,
and therefore using (\ref{1313}) we obtain
\[
\langle\,f\,,\,a\star\om_0\,\rangle \,=\,\langle\,\om_0\,,\,f\star a\,\rangle \,=\,
\langle\,\om_0\,,\,e_0(f\star a)e_0\,\rangle \,=\,0
\]
for all $f\in\LO$, and therefore $a \star \om_0 = 0$. So, it follows
from Lemma \ref{P.faithful} that $\om_0 = 0$. Hence, $e_0 = 0$ and $\LL$ is purely atomic.
\end{proof}

So, by the previous lemma, we conclude that $\LL$ is a direct sum of type I factors:
\begin{equation}\label{06}
\LL = l^\infty-\oplus_{i\in \mathcal{I}}\,\B(H_i)\,.
\end{equation}
Then $\LT$, being the (unique) standard Hilbert space of $\LL$,
can be identified with the Hilbert space $l^2-\oplus_{i\in I}\,\mathcal{HS}(H_i)$,
where $\mathcal{HS}(H_i)$ is the Hilbert--Schmidt space over $H_i$.
Moreover, from the uniqueness of the standard representation,
it follows that the representation of every summand $\B(H_i)$
on $\LT$ is equivalent to their representation on
$H_i\otimes H_i$, mapping $a \in \B(H_i)$ to $a\otimes 1 \in \B(H_i\otimes H_i)$.
In particular, if $a \in \B(H_i)$ is compact on $\LT$, then either $a=0$ or $H_i$ is finite dimensional.

We denote by $1_j\in\oplus_{i\in \mathcal{I}} \B(H_i)$ the projection onto $H_j$.
Then $1_j$ is in the center of $\oplus_{i\in \mathcal{I}} \B(H_i)$, and
$x = \sum_i 1_i x$ for all $x\in \oplus_{i\in \mathcal{I}} \B(H_i)$.

If $a\in\LL\cap \B_0(\LT)$, then $1_j a$ is a compact operator on $\LT$.
Thus, if $1_j a \neq 0$, the Hilbert space $H_j$ is finite dimensional.

\begin{lemma}\label{05}
For each $j\in \mathcal{I}$ there exists $a\in\LL\,\cap\, \B_0(\LT)$ such that $1_j a \neq 0$.
In particular, dim$(H_j) < \infty$ for all $j\in\mathcal{I}$.
\end{lemma}
\begin{proof}
Let $0\leq a\in\LL\,\cap\, \B_0(\LT)$ be non-zero. We show that there exists $f\in\LO$
such that $1_j (a\star f)\neq 0$, and this yields the lemma by (\ref{1313}).
So, suppose in contrary that $1_j (a\star f) = 0$ for all $f\in \LO$, and choose $0\leq \om\in\LO$
such that $\om(1_j)\neq 0$. Then we get
\[
\langle\, f\,,\,(1_j\om)\star a  \,\rangle\,=\,\langle\, \om\,,\,1_j (a\star f)  \,\rangle\,=\,0
\]
for all $f\in\LO$, which implies that $(1_j\om)\star a = 0$. But,
by Lemma \ref{P.faithful} this contradicts our assumptions.
\end{proof}

The proof of the following lemma is standard.
\begin{lemma}\label{big-small}
If $H$ and $K$ are Hilbert spaces, and $\Phi: \B(H) \rightarrow \B(K)$ is an injective
$*$-homomorphism, then $\Phi(x) \in \B_0(K)$ implies $x\in\B_0(H)$.
\end{lemma}

\begin{lemma}\label{010}
We have
\[
c_0-\oplus_{i\in I}\,\B_0(H_i)\,=\,\LL\,\cap\, \B_0(\LT)\,=\,\CZ \,.
\]
\end{lemma}
\begin{proof}
First equality follows from Lemmas \ref{05} and \ref{big-small}.
We first show that the first two spaces are included in $\CZ$.
Suppose that $\phi\in\LL^*$ is zero on
$\CZ$. Denote by $\phi_n$ and $\phi_s$
the normal and singular parts of $\phi$, respectively.
Then, by (\ref{1313}) we have $\phi(a\star f)=0$ for all $a\in \LL\,\cap\, \B_0(\LT)$ and $f\in\LO$,
and since $\phi_s$ is zero on any compact operator, it follows that $\langle\,\phi_n\,,\,a\star f\,\rangle =0$.
This implies that $\phi_n\star a = 0$
for all $a\in \LL\,\cap\, \B_0(\LT) = c_0-\oplus_{i\in I}\,\B(H_i)$.
Since the latter is weak* dense in $\LL$ and convolution map $x\mapsto\phi_n\star x$
is normal on $\LL$, it follows that $\phi_n = 0$. Hence, $\phi$ vanishes on
$\LL\,\cap\, \B_0(\LT)$, and therefore the inclusion follows.

For the reverse inclusion,
suppose that $\mu\in\MG$ is zero on $c_0-\oplus_{i\in I}\,\B(H_i)\,=\,\LL\,\cap\, \B_0(\LT)\,\subseteq\,\CZ$,
then similar to the above we get $\mu\star a = 0$ for all $a\in c_0-\oplus_{i\in I}\,\B(H_i)$,
and since convolution action by $\mu$ is normal on $\LL$, we get $\mu=0$.
This completes the proof.
\end{proof}

\noindent
{\bf Proof of Theorem \ref{main}.}\
From Lemma \ref{010} we have $\CZ = {c_0-\oplus_i \B(H_i)}$ is an ideal in $\CZ^{**} = {l^\infty-\oplus_i \B(H_i)}$.
Hence $\G$ is discrete by \cite[Theorem 4.4]{Run-JOT}.\qed

\begin{corollary}\label{08}
Let $\G$ be a regular locally compact quantum group such that
\[
\LL\,=\, l^\infty-\oplus_{i\in \mathcal{I}}\,\B(H_i)\,.
\]
Then the following are equivalent:
\begin{itemize}
\item[(i)]
dim$(H_i) < \infty$ for all ${i\in \mathcal{I}}$;
\item[(ii)]
dim$(H_i) < \infty$ for some ${i\in \mathcal{I}}$;
\item[(iii)]
dim$(H_i) = 1$ for some ${i\in \mathcal{I}}$;
\item[(iv)]
$\G$ is discrete.
\end{itemize}
\end{corollary}
\begin{proof}
(i)$\Rightarrow$(ii) is trivial. From the structure theory of discrete quantum groups \cite{Wor},
(iv) implies other statements.
The implication (iii)$\Rightarrow$(iv) was proved in \cite[Proposition 4.1]{Run-JOT}
(without regularity condition).
(ii) implies that $\LL\cap\BZ \neq \{0\}$, and hence by Theorem \ref{main}
yields (iv).
\end{proof}

\noindent
{\bf Remark.}\ \
By \cite[Theorem 3.5]{Taylor}, $\LO$ has the Radon--Nikodym property (RNP) if and only if
the von Neumann algebra $\LL$ is purely atomic, i.e. $l^\infty$-direct sum
of type I factors. So, under regularity condition, Corollary \ref{08} gives a distinction between
discreteness of $\G$ and the RNP of $\LO$, based on the dimension of direct summands of $\LL$.

\bibliographystyle{plain}

\end{document}